\documentclass[12pt]{amsart}

\usepackage[margin=1.15in]{geometry}

\usepackage{amsmath,amscd,amssymb,amsfonts,latexsym,wasysym, mathrsfs, mathtools,hhline,color, bm}
\usepackage[all, cmtip]{xy}

\usepackage{url}

\definecolor{hot}{RGB}{65,105,225}

\usepackage[pagebackref=true,colorlinks=true, linkcolor=hot ,  citecolor=hot, urlcolor=hot]{hyperref}

\usepackage{ textcomp }
\usepackage{ tipa }

\usepackage{graphicx,enumerate}

\usepackage{enumitem}

\theoremstyle{plain}
\newtheorem{theorem}{Theorem}[section]
\newtheorem{proposition}[theorem]{Proposition}

\newtheorem{lm}[theorem]{Lemma}

\newtheorem{corollary}[theorem]{Corollary}

\newtheorem{lemma}[theorem]{Lemma}
\newtheorem{thrm}[theorem]{Theorem}

\theoremstyle{definition}

\newtheorem{remark}[theorem]{Remark}

\newtheorem{ex}[theorem]{Example}
\newtheorem*{ex*}{Example}
\newtheorem{problem}{Problem}

\def\be{\begin{equation}}
\def\ee{\end{equation}}

\def\bt{\begin{thrm}}
\def\et{\end{thrm}}

\def\bc{\begin{cor}}
\def\ec{\end{cor}}

\def\br{\begin{rmk}}
\def\er{\end{rmk}}

\def\bp{\begin{prop}}
\def\ep{\end{prop}}

\def\bl{\begin{lm}}
\def\el{\end{lm}}

\def\bex{\begin{ex}}
\def\eex{\end{ex}}

\def\bd{\begin{defn}}
\def\ed{\end{defn}}

\newcommand\sE{{\mathcal E}}

\newcommand\sU{{\mathcal U}}
\newcommand\sV{{\mathcal V}}

\newcommand\sZ{\mathcal{Z}}

\DeclareMathOperator{\Jac}{Jac}

\DeclareMathOperator{\codim}{codim}              
\DeclareMathOperator{\reg}{reg}                  
\DeclareMathOperator{\sing}{sing}                  

\DeclareMathOperator{\Eu}{Eu}

\DeclareMathOperator{\EDdeg}{EDdeg}
\DeclareMathOperator{\PEDdeg}{EDdeg_{proj}}

\def\bC{\mathbb{C}}
\def\RR{\mathbb{R}}

\def\bP{\mathbb{P}}

\def\lra{\longrightarrow}

\def\bZ{\mathbb{Z}}

\def\balpha{{\bm\alpha}}
\def\bbeta{{\bm\beta}}
\def\C{\mathbb{C}}

\title[ED degree of projective varieties]{Euclidean distance degree of projective varieties}

\author{Laurentiu G. Maxim}
\address{Department of Mathematics,         University of Wisconsin-Madison,  480 Lincoln Drive, Madison WI 53706-1388, USA.}
\email {maxim@math.wisc.edu}\urladdr{https://www.math.wisc.edu/~maxim/}
\author{Jose Israel Rodriguez}
\address{Department of Mathematics,         University of Wisconsin-Madison,  480 Lincoln Drive, Madison WI 53706-1388, USA.}
\email {jose@math.wisc.edu}\urladdr{http://www.math.wisc.edu/~jose/}
\author{Botong Wang}
\address{Department of Mathematics,         University of Wisconsin-Madison,  480 Lincoln Drive, Madison WI 53706-1388, USA.}
\email {wang@math.wisc.edu}\urladdr{http://www.math.wisc.edu/~wang/}



\begin{document}

\date{\today}

\maketitle

\begin{abstract}  
We give a positive answer to a conjecture of Aluffi-Harris on the computation of the Euclidean distance degree of a possibly singular projective variety in terms of the local Euler obstruction function.
\end{abstract}



\section{Introduction}\label{intro}
Many models in data science, engineering and other applied fields are realized as real algebraic varieties, for which one needs to solve a {\it nearest point problem}. Specifically, for such a real algebraic variety $X \subset \RR^n$, given $\balpha\in\RR^n$, one needs to compute $\balpha^*\in X_{\reg}$ that minimizes the (squared) Euclidean distance from the given point $\balpha$. (Here, $X_{\reg}$ denotes the nonsingular locus of $X$.)

An algebraic measure of complexity of this optimization problem and a good indicator of the running time needed to solve the problem exactly consists of computing all of the critical points of the squared Euclidean distance function on the (nonsingular part of the) Zariski closure of $X$ in $\bC^n$. This number of complex critical points is an intrinsic invariant of the optimization problem at hand, called the {\it Euclidean distance degree}. It was introduced in \cite{DHOST}, and has been extensively studied since, e.g., see \cite{AH}, \cite{HL}, \cite{OSS}, \cite{MRW}.

In more details, to any $\bm{\alpha}=(\alpha_1, \ldots, \alpha_n)\in \bC^n$, one associates the squared Euclidean distance function  $f_{\bm{\alpha}}: \bC^n\to \bC$ given by
$$f_{\bm{\alpha}}(x_1,\ldots,x_n):=\sum_{1\leq i\leq n}(x_i-\alpha_i)^2.$$
If $X$ is an irreducible closed subvariety of $\bC^n$ then, for generic choices of $\bm{\alpha}$, the function $f_{\bm{\alpha}}|_{X_{\reg}}$ has finitely many critical points on the nonsingular locus $X_{\reg}$ of $X$. This number of critical points is independent of the generic choice of $\bm{\alpha}$, and it is called the {\it Euclidean distance degree} (or ED degree) of $X$, denoted by $\EDdeg(X)$. 

The ED degree of complex affine varieties was recently computed in \cite{MRW}, and it was used to solve the {\it multiview conjecture} of \cite[Conjecture 3.4]{DHOST}. Specifically, the following result holds:

\bt[\cite{MRW}]\label{eda}
Let $X$ be an irreducible closed subvariety of $\bC^n$. Then for a general $\bbeta=(\beta_0,\beta_1,\ldots,\beta_n)\in \bC^{n+1}$ we have:
\be\label{e1i} \EDdeg(X)=(-1)^{\dim X}\chi\big({\rm Eu}_X|_{U_\bbeta}\big),\ee
where $$U_\bbeta:=\bC^n\setminus \{\sum_{1\leq i\leq n}(x_i-\beta_i)^2+\beta_0=0\},$$ and ${\rm Eu}_X$ is the local Euler obstruction function on $X$. In particular, if $X$ is a smooth closed subvariety of $\bC^n$, then for general $\bbeta\in \bC^{n+1}$, we get:
\be\label{e2i} \EDdeg(X)=(-1)^{\dim X}\chi\big(X\cap U_\bbeta\big).\ee
\et 

On the other hand, many models are realized as affine cones because the varieties are defined by homogeneous polynomials. 
In this setting it is natural to consider the model as a projective variety. 
Examples of such models occur in structured low rank matrix approximation \cite{OSS}, 
low rank tensor approximation, formation shape control \cite{AH2014}, and all across algebraic statistics \cite{DSS2009,Sullivant2018}. 
Thus, if $X$ is an irreducible closed subvariety of $\bP^{n}$, we define the {\it (projective) Euclidean distance degree} of $X$ by $$\PEDdeg(X)=\EDdeg (C(X)),$$ where $C(X)$ is the affine cone of $X$ in $\bC^{n+1}$.
Here, we address the following problem:
\begin{problem}
If $X$ is a projective variety of $\bP^n$ with its affine cone denoted by $C(X)$,  then give a description of the ED degree of $C(X)$ in terms of the topology of $X$.
%
\end{problem}

The motivation for studying the ED degree problem in terms of projective geometry 
rather than working with affine cones comes from the fact that an affine cone acquires a complicated singularity at the cone point.

The above problem has been recently considered by Aluffi-Harris in \cite{AH}, building on preliminary results from \cite{DHOST}. The main result of Aluffi-Harris can be formulated as follows:
\bt\cite[Theorem 8.1]{AH} \label{edpsm}
Let $X$ be a smooth subvariety of $\bP^n$, and assume that $X \nsubseteq Q$, where 
$Q=\{[x_0:\ldots:x_n]\in \bP^n \mid x_0^2+x_1^2+\ldots + x_n^2=0\}$ is the isotropic quadric in $\bP^n$. Then
\be\label{e3i}
\PEDdeg(X)=(-1)^{\dim X}\chi\big(X \setminus (Q \cup H)\big),
\ee
where $H$ is a general hyplerplane.
\et
Theorem \ref{edpsm} is proved in \cite{AH} by using the theory of characteristic classes for singular varieties, and it provides a generalization of \cite[Theorem 5.8]{DHOST}, where it was assumed that the smooth projective variety $X$ intersects the isotropic quadric $Q$ transversally, i.e., that $Q \cap X$ is a smooth hypersurface of $X$. Aluffi and Harris also expressed hope that formula (\ref{e3i}) would admit a more direct proof, without reference to characteristic classes, which may be more amenable to generalization. In fact, they conjectured that formula (\ref{e3i}) should admit a natural generalization to arbitrary (possibly singular) projective varieties by using the ``Euler-Mather characteristic'' defined in terms of the local Euler obstruction function. We address their conjecture in the following statement:
\bt\label{edpsi}
Let $X$ be any irreducible closed subvariety of $\bP^n$. Then for a general $\bbeta=(\beta_0,\beta_1,\ldots,\beta_n)\in \bC^{n+1}\setminus\{0\}$ we have
\be\label{e4i} \PEDdeg(X)=(-1)^{\dim X}\chi\big({\rm Eu}_X|_{\mathcal{U}_\bbeta}\big),\ee
where $$ \mathcal{U}_\bbeta:=\bP^n \setminus (Q \cup H_{\bbeta}),$$
with $Q$ denoting the isotropic quadric and $H_{\bbeta}:= \{\beta_0 x_0+\beta_1 x_1+\ldots + \beta_nx_n=0 \}.$
\et

The proof of Theorem \ref{edpsi} is Morse-theoretic in nature (resting on results of \cite{STV,TS}), and it employs ideas similar to those needed for proving Theorem \ref{eda}. While one may be tempted to deduce Theorem \ref{edpsi} directly from Theorem \ref{eda}, such an approach proves to be surprisingly challenging due to the presence of the cone point.

Note that in the case when $X\subset \bP^n$ is a smooth subvariety, Theorem \ref{edpsi} reduces to the statement of Theorem \ref{edpsm}. Indeed, if $X$ is smooth, then $\Eu_X=1_X$ and the assertion follows. Theorem \ref{edpsi} also generalizes \cite[Proposition 3.1]{AH}, where the ED degree of a projective variety $X \subset \bP^n$  is computed under the assumption that $X$ intersects the isotropic quadric $Q$ transversally.

 
 \medskip

The paper is organized as follows. 
In Section~\ref{sec:criticalPoints}, we introduce the local Euler obstruction function and we recall its main properties. 
In Section~\ref{sec:pdeg}, we prove our main Theorem \ref{edpsi}, after first interpreting the projective ED degree as the number of the degeneration points of a certain rational $1$-form. Section \ref{sec:ex} is devoted to computations of the projective ED degree on specific examples of singular projective varieties.

\medskip 

{\bf Acknowledgements}
 The authors thank J\"org Sch\"urmann for useful discussions. 
L. Maxim is partially supported by the Simons Foundation Collaboration Grant \#567077 and by the
Romanian Ministry of National Education, CNCS-UEFISCDI, grant PN-III-P4-ID-PCE-2016-0030. 
J.~I. Rodriguez is partially supported by the College of Letters and Science, UW-Madison. 
B. Wang is partially supported by the NSF grant DMS-1701305.


\section{Local Euler Obstruction and critical points of generic linear functions}\label{sec:criticalPoints}


Let $X$ be a complex algebraic variety. It is well-known that such a variety can be endowed with a Whitney stratification. Roughly speaking, this means that $X$ admits a partition $\mathcal{S}$  into nonempty, locally closed nonsingular subvarieties (called {\it strata}), along which $X$ is topologically equisingular.


A function $\varphi:X \to \bZ$ is {\it constructible} with respect to a given Whitney stratification $\mathcal{S}$ (or, $\mathcal{S}$-constructible), if $\varphi$ is constant along each stratum $S \in \mathcal{S}$. The {\it Euler characteristic} of an $\mathcal{S}$-constructible function $\varphi$ is the Euler-Poincar\'e characteristic of $X$ weighted by $\varphi$, that is,
\be \chi(\varphi):=\sum_{S \in \mathcal{S}} \chi(S) \cdot \varphi(S),
\ee
where $\varphi(S)$ denotes the (constant) value of $\varphi$ on the stratum $S$.

A fundamental role in the formulation of our main result is played by the {\it local Euler obstruction} function $$\Eu_X:X \to \bZ,$$
which is an essential ingredient in MacPherson's definition of Chern classes for singular varieties \cite{M74}. 
The precise definition of the local Euler obstruction function is not needed in this paper, but see, e.g., \cite[Section 4.1]{Di} for an introduction. We only list here properties of the local Euler obstruction function that are relevant in this paper:
\begin{enumerate}[label=(\roman{*}), ref=\roman{*}]
\item $\Eu_X$ is constant along the strata of a fixed Whitney stratification $\mathcal{S}$ of $X$, i.e., $\Eu_X$ is $\mathcal{S}$-constructible.
\item If $x \in X$ 	is a smooth point then $\Eu_X(x)=1$.
\item If $X=\cup_i X_i$ is the decomposition of $X$ into irreducible components, then $\Eu_X(x)=\sum \Eu_{X_i}(x)$, where the sum is over all irreducible components that pass through the point $x$.
\item If $(X,x)$ is an isolated singularity germ, then $\Eu_X(x)=\chi(CL(X,x)),$ where $CL(X,x)$ is the complex link of $x$ in $X$.
\item If $X$ is a curve, then $\Eu_X(x)$ is the multiplicity of $X$ at $x$.
\item\label{p6} The Euler obstruction function is preserved under generic hyperplane sections. More precisely, suppose $X$ is a projective variety. If $S\subset X$ is a stratum in a fixed Whitney stratification of $X$ of positive dimension, then for a general hyperplane $H$, the value of $\Eu_X$ along $S$ is equal to the value of $\Eu_{X\cap H}$ along $S\cap H$. 
\item The Euler obstruction function is an analytic invariant. More precisely, given two varieties $X$ and $X'$, and points $x\in X$, $x'\in X'$, if the analytic germ of $X$ at $x$ is isomorphic to the analytic germ of $X'$ at $x'$, then $\Eu_X(x)=\Eu_{X'}(x')$. In particular, if $U$ is a Zariski open set in $X$, then $\Eu_U=\Eu_X|_{U}$. 
\end{enumerate}

We should also point out that the local Euler obstruction function is not motivic, in the sense that, if $Y$ is a closed subvariety of $X$, then in general one has: \be\label{ad}\Eu_X\neq\Eu_Y + \Eu_{X \setminus Y}.\ee For example, just consider $X$ a singular curve with only one singular point $Y$ which is a double point. However, if $Y$ is a generic hyperplane section of $X$, then one has an equality in (\ref{ad}) by property (\ref{p6}). This fact is used in the following result, see \cite[Equation~(2)]{STV} and also \cite[Theorem 1.2]{TS}:
\begin{theorem}\cite{STV}\cite{TS}\label{theorem_top_singular}
	Let $X$ be an irreducible  closed subvariety in $\C^n$. Let $l: \C^n\to \C$ be a general linear function, and let $H_c$ be the hyperplane in $\C^n$ defined by the equation $l=c$ for a general $c\in \C$. Then
	the number of critical points of $l|_{X_{\reg}}$ is equal to $(-1)^{\dim X}\chi({\rm Eu}_{X}|_{U_c})$, where $U_c=X\setminus H_{c}$ and ${\rm Eu}_X$ is the local Euler obstruction function on $X$.
\end{theorem}

\begin{remark}\label{remark_transversal}
	The condition of being general in the above theorem can be made precise as follows. The linear function $l$ and the number $c$ are general, if $\overline{H_c}$ in $\bP^n$ intersects $\overline{X}$ as well as $\overline{X}\setminus X$ transversally in the stratified sense, where $\overline{H_c}$ and $\overline{X}$ denote the closure of $H_c$ and $X$ in $\bP^n$, respectively. 
\end{remark}

\section{Euclidean distance degree of a projective variety}\label{sec:pdeg}

Let $X$ be an irreducible closed subvariety of $\bP^{n}$.
Recall from the Introduction that the {\it (projective) Euclidean distance degree} of $X$, denoted here by $\PEDdeg(X)$ (to emphasize the fact that $X$ is projective), is defined as 
\[
\PEDdeg(X)=\EDdeg(C(X)),
\]
where $C(C)$ is the affine cone on $X$ in $\bC^{n+1}$.

The first result of this section, Proposition \ref{prop_PEDequal}, gives am interpretation of the ED degree of a projective variety in terms of the number of the degeneration points of a certain rational $1$-form. Let $[x_0: x_1: \ldots: x_n]$ be the projective coordinates of $\bP^n$. Given any $\bbeta=(\beta_0, \beta_1, \ldots, \beta_n)\in \bC^{n+1}\setminus \{0\}$, we define the rational function $h_\bbeta$ on $\bP^n$ by
$$h_\bbeta:=\frac{(\beta_0 x_0+\beta_1 x_1+\cdots+\beta_n x_n)^2}{x_0^2+x_1^2+\cdots+x_n^2},$$
and we define $$\sU_{\bbeta}:=\bP^n\setminus (Q \cup H_{\bbeta})$$ to be the complement of the \emph{isotropic quadric}
$Q:=\{[x_0:\ldots:x_n]\mid x_0^2+x_1^2+\cdots+x_n^2=0\}$ and the hyperplane $H_{\bbeta}:=\{[x_0:\ldots:x_n]\mid \beta_0 x_0+\beta_1 x_1+\cdots+\beta_n x_n=0\}$ in $\bP^n$. With the above notation, we have the following:
\begin{proposition}\label{prop_PEDequal}
Let $X$ be a proper nonempty irreducible subvariety of $\bP^{n}$ which is not contained in the isotropic quadric  $Q$.
For general $\bbeta\in \bC^{n+1}$, the number of the degeneration points of $(d\log h_\bbeta)|_{X_{\reg}}$ is equal to $\PEDdeg(X)$. 
\end{proposition}

\newcommand{\bq}{\mathbf{q}} %
\newcommand{\bv}{\mathbf{v}} %
\newcommand{\bx}{\mathbf{x}} %
\newcommand{\bqbar}{\bar\bq}
\newcommand{\bxbar}{\bar\bx}

\begin{proof}
The assertion is a consequence of the following two statements:
\begin{enumerate}
\item\label{item_eq1} For general $\bbeta$, none of the critical points of $f_{\bbeta}|_{C(X)_{\reg}}$ lies in the hyperplane $
H_{\bbeta}=\{ [x_0:\dots:x_n] : \beta_0 x_0+\beta_1 x_1+\cdots+\beta_n x_n=0\}$, where 
\[ 
f_\bbeta(x_0,x_1,\ldots,x_n)=\sum_{i=0}^n(x_i-\beta_i)^2
\]
is the squared distance function. 
\item\label{item_eq2} For general $\bbeta$, a point $\bqbar=[q_0: q_1: \ldots: q_n]\in \bP^{n}$ is a degeneration point of $(d\log h_\bbeta)|_{X_{\reg}}$ if and only if there exists a lifting $\bq=(q_0, q_1, \ldots, q_n)\in \bC^{n+1}$ of $\bqbar$ that is a critical point of $f_{\bbeta}|_{C(X)_{\reg}}$. 
\end{enumerate}
Notice that for general $\bbeta$, there does not exist a pair of critical points of $f_{\bbeta}|_{C(X)_{\reg}}$ that differ by a nonzero scalar multiplication. Thus, to prove the proposition, it suffices to show the above two statements. 

Define the closed subvariety $\sZ$ of $C(X)_{\reg}\times \bC^{n+1}$ consisting of points $(y, \bbeta)$ such that $y$ is a critical point of $f_\bbeta|_{C(X)_{\reg}}$\footnote{The variety $\sZ$ is an open subvariety of the ED correspondence variety $\sE_X$ defined in \cite{DHOST}.}. By definition, the subvariety $\sZ$ can be realized as the total space of a vector bundle over $C(X)_{\reg}$ through the projection to first factor, where the fiber dimension is equal to $\codim C(X)$. Thus, $\sZ$ is irreducible and of dimension $n+1$. Let $\sZ_0\subset \sZ$ be the closed locus of $\sZ$ corresponding to the critical points that lie in the hyperplane $\beta_0 x_0+\beta_1 x_1+\cdots+\beta_n x_n=0$. Then $\sZ_0$ is a Zariski closed subset of $\sZ$. 

Suppose that, for general $\bbeta$, there exists at least one critical points of $f_{\bbeta}|_{C(X)_{\reg}}$ that lies in the hyperplane $\beta_0 x_0+\beta_1 x_1+\cdots+\beta_n x_n=0$. Then $\dim \sZ_0= n+1$. Since $\sZ$ is irreducible, we have $\sZ_0=\sZ$. Therefore, every critical point of $f_{\bbeta}|_{C(X)_{\reg}}$ lies in the hyperplane $\beta_0 x_0+\beta_1 x_1+\cdots+\beta_n x_n=0$. Notice that for any $\bq\in C(X)_{\reg}$ and any normal vector $\bv=(v_0, v_1, \ldots, v_n)$ to $C(X)$ at $\bq$, the function 
\[
f_{\bq+\bv}(x_0,x_1,\dots,x_n)=\sum_{0\leq i\leq n}(x_i-q_i-v_i)^2
\]
 has a critical point at $\bq$. Thus, we have
 $\bq$ is in the hyperplane defined by 
   $$(q_0+v_0)x_0+(q_1+v_1)x_1+\cdots+ (q_n+v_n)x_n=0,$$
hence 
$$(q_0+v_0)q_0+(q_1+v_1)q_1+\cdots+ (q_n+v_n)q_n=0.$$
Since any scalar multiple of $\bv$ is also a normal vector to $C(X)$ at $\bq$, the above equality implies that 
$$q_0^2+q_1^2+\cdots+q_n^2=0 \text{ and } q_0 v_0+q_1 v_1+\cdots q_n v_n=0.$$
This contradicts the assumption that $X$ is not contained in the isotropic quadric defined by 
$x_0^2+x_1^2+\cdots+x_n^2=0$. Thus, we have proved (\ref{item_eq1}).

Notice that
\begin{equation}\label{eq_dlog}
d\log h_\bbeta=\frac{2\beta_0 dx_0+2\beta_1 dx_1+\cdots+2\beta_n dx_n}{\beta_0 x_0+\beta_1 x_1+\cdots+\beta_n x_n}-\frac{2x_0 dx_0+2x_1dx_1+\cdots+2x_ndx_n}{x_0^2+x_1^2+\cdots+x_n^2}
\end{equation}
which can also be considered as a $\mathbb{C}^*$-equivariant form on $\mathbb{C}^{n+1}\setminus \{0\}$.
 If $\bqbar\in \bP^n$ is a degeneration point of $(d\log h_\bbeta)|_{Y_{\reg}}$, then at any lifting $\bq\in\mathbb{C}^{n+1}\setminus \{0\}$ of $\bqbar$,
 the two cotangent vectors 
$$x_0dx_0+x_1dx_1+\cdots+x_ndx_n \text{ and }\beta_0dx_0+\beta_1 dx_1+\cdots+\beta_n dx_n$$
differ by a unique nonzero scalar. 
Therefore
there exists a unique lifting $\bq\in\mathbb{C}^{n+1}\setminus\{0\}$,
such that 
$$x_0dx_0+x_1dx_1+\cdots+x_ndx_n=\beta_0dx_0+\beta_1 dx_1+\cdots+\beta_n dx_n,$$ 
as cotangent vectors of $C(X)$ at $\bq$, 
or equivalently $\bq$ is a critical point of 
$f_\bbeta |_{C(X)_{\reg}}.$

Conversely, suppose $\bq\in \bC^{n+1}$ is a critical point of $f_\bbeta|_{C(X)_{\reg}}$, that is
\begin{equation}\label{eq_degenerate}
q_0dx_0+q_1dx_1+\cdots+q_ndx_n=\beta_0dx_0+\beta_1 dx_1+\cdots+\beta_n dx_n
\end{equation}
 as cotangent vectors of $C(X)_{\reg}$ at $\bq$. 
We need to show that the $\bC^*$-equivariant form $(d\log h_\bbeta)|_{C(X)_{\reg}}$ degenerates on the line $t\bq=(tq_0,tq_1,\ldots,tq_n)$ parametrized by $t$. Notice that the restriction of $d\log h_\bbeta$ to any line passing through the origin is zero. Thus, it suffices to show that the further restriction of the $1$-form $(d\log h_\bbeta)|_{C(X)_{\reg}\cap H}$ has a degeneration point at $\bq$, where $H$ denotes the hyperplane defined by 
$$
H=\{(x_0, x_1, \ldots, x_n)\in \bC^{n+1}\mid \beta_0(x_0-q_0)+\beta_1(x_1-q_1)+\cdots+\beta_n(x_n-q_n)=0\}. 
$$
Here, notice that by part (\ref{item_eq1}), 
the hyperplane $H$ does not pass through the origin, and hence intersects $C(X)_{\reg}$ transversally. 
When restricting to $C(X)_{\reg}\cap H$, 
the cotangent vector $\beta_0dx_0+\beta_1dx_1+\cdots+\beta_ndx_n$ vanishes. 
Thus, by (\ref{eq_degenerate}), the cotangent vector
 $q_0dx_0+q_1dx_1+\cdots+q_ndx_n$ 
 of $C(X)_{\reg}\cap H$ also vanishes at $\bq$. 
  Therefore, $(d\log h_\bbeta)|_{C(X)_{\reg}\cap H}$ degenerates at $\bq$. 
  Thus, $(d\log h_\bbeta)|_{C(X)_{\reg}}$ degenerates along the line $t\bq$, and equivalently the $1$-form $(d\log h_\bbeta)|_{X_{\reg}}$ degenerates at $\bqbar$. 
\end{proof}

Let us now denote the complement of the isotropic quadric $Q=\{x_0^2+x_1^2+\cdots+x_n^2=0\}$ in $\bP^n$ by $U:=\bP^n \setminus Q$. Denote the affine variety $\{x_0^2+x_1^2+\cdots+x_n^2=1\}\subset \C^{n+1}$ by $V$. 
\begin{lemma}\label{lemma_cover}
	The following algebraic map 
	$$\Phi: V\to U, \quad (x_0, x_1, \ldots, x_n)\mapsto [x_0: x_1:\ldots: x_n]$$
	is a double covering map. 
\end{lemma}
\begin{proof}
	The map factors through the quotient map $V\to V/\{\pm 1\}$, and is evidently a set-theoretic 2-to-1 map. It is straightforward to check that $\Phi$ induces isomorphism on the tangent spaces.
\end{proof}

Let $Y:=\Phi^{-1}(X\cap U)$. The following two corollaries are immediate consequences of the preceding lemma. 
\begin{corollary}\label{cor_1}
	The number of the degeneration points of $(d\log l_\bbeta)|_{Y_{\reg}}$ is twice the number of the degeneration points of $(d\log h_\bbeta)|_{X_{\reg}}$, where $l_\bbeta=\beta_0x_0+\beta_1 x_1+\cdots+\beta_n x_n$. 
\end{corollary}
\begin{proof}
	Notice that, as meromorphic functions, 
	$$h_\bbeta\circ \Phi=\frac{(\beta_0 x_0+\beta_1 x_1+\cdots+\beta_n x_n)^2}{x_0^2+x_1^2+\cdots+x_n^2}=l_\bbeta^2$$
	on $V$. Therefore, $(d\log l_\bbeta)|_{Y_{\reg}}$ and $\Phi^*(d\log h_\bbeta)|_{X_{\reg}}$ only differ by a nonzero scalar. Thus the assertion follows from Lemma \ref{lemma_cover}. 
\end{proof}
\begin{corollary}\label{cor_2}
	Under the above notations, 
	\begin{equation}\label{eq_2}
	\chi({\rm Eu}_{Y}|_{\sV_\bbeta})=2\chi({\rm Eu}_X|_{\sU_\bbeta})
	\end{equation}
	where $\sV_\bbeta$ is the complement of the hyperplane $\{l_\bbeta=0\}$ in $\C^{n+1}$. 
\end{corollary}
\begin{proof}
	The double covering map $\Phi: V\to U$ induces a double covering map $Y\cap \sV_\bbeta\to X\cap \sU_\bbeta$. Since the Euler obstruction function is an analytic invariant, the pullback of ${\rm Eu}_X|_{\sU_\bbeta}$ is equal to ${\rm Eu}_{Y}|_{\sV_\bbeta}$. Hence equation (\ref{eq_2}) follows. 
\end{proof}

\begin{proposition}\label{prop_Y}
	In the above notations, the number of the critical points of $l_\bbeta|_{Y_{\reg}}$ is equal to $(-1)^{\dim Y}\chi({\rm Eu}_Y|_{\sV_\bbeta})$. 
\end{proposition}
\begin{proof}
	By Theorem \ref{theorem_top_singular} and Remark \ref{remark_transversal}, it suffices to show that the closure of $\{l_\bbeta=0\}$ in $\bP^{n+1}$ intersects $\overline{Y}$ as well as $\overline{Y}\setminus Y$ transversally. As $\bbeta$ varies, the closure of $\{l_\bbeta=0\}$ form divisors in a linear system. Since the linear system has only one base point, which is the origin of $\bC^{n+1}$, and since $Y$ is contained in the affine hypersurface $\{z_0^2+z_1^2+\cdots+z_n^2=1\}$, the transversal condition follows from Bertini's theorem.
\end{proof}

\begin{proof}[Proof of Theorem \ref{edpsi}]
	By Proposition \ref{prop_PEDequal}, it suffices to show that the number of the degeneration points of $(d\log h_\bbeta)|_{X_{\reg}}$ is equal to $(-1)^{\dim X}\chi\big({\rm Eu}_X|_{\mathcal{U}_\bbeta}\big)$. By Corollary \ref{cor_1} and Corollary \ref{cor_2}, it suffices to show the following equality:
	\begin{equation}\label{eq_last}
	\#\big\{\text{degeneration points of } (d\log l_\bbeta)|_{Y_{\reg}}\big\} =  (-1)^{\dim Y}\chi({\rm Eu}_Y|_{\sV_\bbeta}).
	\end{equation}
	By the proof of Proposition \ref{prop_Y}, we know that $\{l_\bbeta=0\}$ intersects $Y$ transversally, which implies that $0$ is not a critical value of $l_\bbeta|_{Y_{\reg}}$. Therefore, the number of the degeneration points of $(d\log l_\bbeta)|_{Y_{\reg}}$ is equal to the number of the critical points of $l_\bbeta|_{Y_{\reg}}$, and hence equality (\ref{eq_last}) holds. This completes the proof of the theorem. 	
\end{proof}


\section{Examples}\label{sec:ex}
In this section, we show by examples how to compute the ED degree of singular projective varieties by using Theorem \ref{edpsi}.  Throughout this section, we denote the isotropic quadric by $Q$,
 a general hyperplane defined by    $\beta_0 x_0+\beta_1 x_1+\ldots + \beta_nx_n=0$ by
$H_{\bbeta}$, and we set $\sU_\bbeta=\bP^n\setminus(Q\cup H_\bbeta)$.

\begin{ex}
	Let $X$ be the nodal curve in $\bP^2$ defined by $$x_0^2 x_2-x_1^2(x_1+x_2)=0.$$ The only singular point of $X$ is $[0:0:1]$. For curves, the value of Euler obstruction function is equal to the multiplicity. Therefore, $\Eu_X$ is equal to $1$ on the smooth locus $X_{\reg}$ of $X$, and it equals $2$ at the singular point. It is straightforward to check that $X$ intersects the isotropic quadric $Q$ transversally at $6$ points. Since $X$ is of degree $3$, it intersects $H_\bbeta$ transversally at $3$ points. Notice that $X_{\reg}$ is isomorphic to $\bC^*$. Therefore, by inclusion-exclusion, we get:
    $$\chi\big(X_{\reg}\cap \sU_\bbeta\big)=0-9=-9.$$
    Thus, 
    $$\chi(\Eu_X|_{\sU_\bbeta})=(-9)+2=-7$$
    and hence by Theorem \ref{edpsi} we get $\PEDdeg(X)=7$.
\end{ex}

\begin{ex}
	Let $X$ be the nodal curve  in $\bP^2$ defined by $$x_0^2 x_1-\big(x_1-\sqrt{-1}x_2\big)^2 x_2=0.$$ The singular point of $X$ is $[0 : 1 : \sqrt{-1}]$, which is contained in the isotropic quadric $Q$. Then $X\cap Q$ consists of $5$ points, and $X\cap H_\bbeta$ consists of $3$ points. Therefore, 
	$X\cap \sU_\bbeta$ is smooth and $\chi(X\cap \sU_\bbeta)=7$. Thus, $\PEDdeg(X)=7$ by Theorem \ref{edpsi}. 
\end{ex}

\begin{ex}
Let $X$ be the nodal curve in $\bP^2$ defined by $$x_0^3-(\sqrt{-1}x_0^2+x_1^2)x_2=0.$$ Its singular point is $[0:0:1]$, which is not on $Q$. The two curves $X$ and $Q$ intersects tangentially at $[1:0: -\sqrt{-1}]$, and they intersect transversally at $4$ more points. We therefore get $$\chi(\Eu_X|_{\sU_\bbeta})=\chi(X\cap \sU_\bbeta)+1=-7+1=-6.$$ Thus, $\PEDdeg(X)=6$ by Theorem \ref{edpsi}. 
\end{ex}

\begin{ex}
	Let $X\subset \bP^3$ be the surface defined by $$x_0^2x_1-x_2x_3^2=0.$$ The singular locus $X_{\sing}$ of $X$ is defined by $x_0=x_3=0$. On either of the affine charts $x_1\neq 0$ and $x_2\neq 0$, $X$ is isomorphic to the Whitney umbrella $\{x^2=y^2z\}\subset \bC^3$. It is well-known that the Whitney umbrella $\{x^2=y^2z\}\subset \bC^3$ has a Whitney stratification with three strata: $\{(0, 0, 0)\}$, $\{x=y=0\}\setminus\{(0, 0, 0)\}$ and $\{x^2=y^2z\}\setminus \{x=y=0\}$. Therefore, $X$ has a Whitney stratification with three strata $S_3:=\{[0:1:0:0], [0:0:1:0]\}$, $S_2:=\{x_0=x_3=0\}\setminus S_3$ and $S_1=X\setminus \{x_0=x_3=0\}$. 
	
	The Euler obstruction function of the Whitney umbrella is well-known (see e.g. \cite[Example 4.3]{RW18}). More precisely, $\Eu_X$ has value $1, 2$ and 1 along $S_1$, $S_2$ and $S_3$, respectively. Therefore, 
\begin{equation}\label{eq_twosum}
\chi(\Eu_X|_{\sU_\bbeta})=\chi(X\cap \sU_\bbeta)+\chi(S_2\cap \sU_\bbeta).
\end{equation}
We first compute $\chi(X\cap \sU_\bbeta)$.  By the inclusion-exclusion principle, we have: 
\begin{equation}\label{eq_sum}
\chi(X\cap \sU_\bbeta)=\chi(X)-\chi(X\cap Q)-\chi(X\cap H_\bbeta)+\chi(X\cap Q\cap H_\bbeta).
\end{equation}
Notice that we can define a $\bC^*$-action on $X$ by: 
$$t\cdot [x_0:x_1:x_2:x_3]\longmapsto [x_0:tx_1:tx_2:x_3].$$
 The fixed point locus of this action is $\{x_1=x_2=0\}\cup \{x_0=x_3=0\}$. Therefore, by the localization principle, we get that
$$\chi(X)=\chi\big( \{x_1=x_2=0\}\cup \{x_0=x_3=0\} \big)=4.$$
Consider now the projection $p: \bP^3\dashrightarrow\bP^2$ defined by $[x_0:x_1:x_2:x_3]\mapsto [x_1:x_2:x_3]$. The indeterminate locus of the above projection is the point $[1;0;0;0]$, which is not contained in $X\cap Q$. Therefore, $p$ restricts to a regular map $$p_{X\cap Q}: X\cap Q\lra \bP^2.$$ The image of $p_{X\cap Q}$ is equal to $\{(x_1^2+x_2^2+x_3^2)^2x_1+x_2x_3^2=0\}\subset \bP^2$, which is a smooth cubic curve. The ramification locus is defined by $x_0=0$, and hence the map $p_{X\cap Q}$ ramifies over $\{x_1^2+x_2^2+x_3^2=(x_1^2+x_2^2+x_3^2)^2x_1+x_2x_3^2=0\}$. So $p_{X\cap Q}$ ramifies over the $4$ points $[1:\pm\sqrt{-1}:0]$ and $[1:0:\pm\sqrt{-1}]$. Therefore, $X\cap Q$ is a degree two cover of an elliptic curve with $4$ ramification points.  Hence,
$$\chi(X\cap Q)=-4.$$
The intersection $X\cap H_\bbeta$ is a plane nodal cubic, hence
$$\chi(X\cap H_\bbeta)=1.$$
By B\'ezout's theorem, $X\cap Q\cap H_\bbeta$ consists of $6$ points, and hence
$$\chi(X\cap Q\cap H_\bbeta)=6.$$
Plugging everything back in equation (\ref{eq_sum}), we have \be\label{ref1}\chi(X\cap \sU_\bbeta)=4-(-4)-1+6=13.\ee 
Let us next compute $\chi(S_2\cap \sU_\bbeta)$. Notice that the closure of $S_2$ is equal to the line $S_2\cup S_3=\{x_0=x_3=0\}$ in $\bP^3$. This line intersects $Q\cup H_\bbeta$ at $3$ points, which are disjoint from $S_3$. Thus, $S_2$ is isomorphic to $\bP^1$ minus $5$ points. Hence,
\be\label{ref2} \chi(S_2\cap \sU_\bbeta)=-3.\ee
Plugging (\ref{ref1}) and (\ref{ref2}) into equation (\ref{eq_twosum}), we get by Theorem \ref{edpsi} that
 $$\PEDdeg(X)=\chi(\Eu_X|_{\sU_\bbeta})=10.$$

\end{ex}

\begin{remark}
These examples are checked using techniques from symbolic computation \cite{CLO2015} and numerical algebraic  geometry \cite{bertinibook} using 
\texttt{Macaulay2} \cite{M2} and $\texttt{Bertini}$ \cite{Bertini4M2,bertinibook}.
Our code is available at the following directory. 
\begin{quote}
\url{https://www.math.wisc.edu/~jose/r/ComputingEDDegree.zip}
\end{quote}

\end{remark}

\bibliographystyle{abbrv}
\bibliography{REF_ED_Degree}

\end{document}